\documentclass[12pt,twoside]{article}
\usepackage[left=1.5cm,right=1.5cm,top=3cm,bottom=2cm]{geometry}
\usepackage{paralist}
\usepackage{marginnote}
\usepackage{amsfonts}
\usepackage{enumitem}
\usepackage{tikz}
\usetikzlibrary{matrix}
\usepackage{amsthm}
\usepackage{amsrefs}
\usepackage{amssymb}
\usepackage{amsmath}
\usepackage{graphicx}
\usepackage{varwidth}
\usepackage{hyperref}
\usepackage{fancyhdr}
\usepackage{stmaryrd}
\usepackage{multirow}
\usepackage{cleveref}
\usepackage{mathtools}
\usepackage{comment}

\usepackage{blkarray}
\usepackage[autostyle]{csquotes}
\usepackage[mathscr]{euscript}
\hypersetup{
    colorlinks=true,
    linkcolor=black,
    filecolor=magenta,      
    urlcolor=black,
    citecolor=black
}
\DeclareMathAlphabet{\mathcalligra}{T1}{calligra}{c}{h}
\providecommand{\U}[1]{\protect\rule{.1in}{.1in}}
\newtheorem{theorem}{Theorem}[section]
\newtheorem{proposition}[theorem]{Proposition}
\newtheorem{lemma}[theorem]{Lemma}

\newtheorem{remark}[theorem]{Remark}
\let\oldremark\remark
\renewcommand{\remark}{\oldremark\normalfont}

\newtheorem{example}[theorem]{Example}
\let\oldexample\example
\renewcommand{\example}{\oldexample\normalfont}

\newtheorem{examples}[theorem]{Examples}
\let\oldexamples\examples
\renewcommand{\examples}{\oldexamples\normalfont}
\newtheorem{definition}[theorem]{Definition}

\newcommand{\R}{\mathbb{R}}
\newcommand{\N}{\mathbb{N}}
\newcommand{\vf}{\varphi}
\newcommand{\ve}{\varepsilon}

\def\<{{\langle}}
\def\>{{\rangle}}

\def\bea{\begin{eqnarray*} }
\def\eea{\end{eqnarray*} }
\def\be{\begin{equation} }
\def\ee{\end{equation} }

\def\qed{\ifhmode\unskip\nobreak\fi\ifmmode\ifinner
\else\hskip5 pt \fi\fi\hbox{\hskip5 pt \vrule width4 pt
height6 pt  depth1.5 pt \hskip 1pt }}

\DeclareMathOperator*{\diver}{div}
\DeclareMathOperator*{\Vol}{Vol}
\DeclareMathOperator*{\Area}{Area}
\DeclareMathOperator*{\vol}{vol}

\DeclareMathOperator*{\diam}{diam}

\DeclareMathOperator*{\supp}{supp}
\DeclareMathOperator*{\res}{res}

\DeclareMathOperator*{\grad}{grad}

\DeclareMathOperator*{\esssup}{ess\,sup}
\DeclareMathOperator*{\essinf}{ess\,inf}

\DeclareMathOperator*{\dy}{dy}
\DeclareMathOperator*{\tr}{tr}
\DeclareMathOperator*{\ad}{ad}
\begin{document}

\title{On the Steklov spectrum of covering spaces and total spaces}
\author{Panagiotis Polymerakis}
\date{}

\maketitle

\renewcommand{\thefootnote}{\fnsymbol{footnote}}
\footnotetext{\emph{Date:} \today} 
\renewcommand{\thefootnote}{\arabic{footnote}}

\renewcommand{\thefootnote}{\fnsymbol{footnote}}
\footnotetext{\emph{2010 Mathematics Subject Classification.} 58J50, 35P15, 53C99.}
\renewcommand{\thefootnote}{\arabic{footnote}}

\renewcommand{\thefootnote}{\fnsymbol{footnote}}
\footnotetext{\emph{Key words and phrases.} Bottom of spectrum, Dirichlet-to-Neumann map, manifold with bounded geometry, Steklov spectrum, Riemannian covering, amenable covering, Riemannian submersion, Riemannian principal bundle, amenable Lie group.}
\renewcommand{\thefootnote}{\arabic{footnote}}

\begin{abstract}
We show the existence of a natural Dirichlet-to-Neumann map on Riemannian manifolds with boundary and bounded geometry, such that the bottom of the Dirichlet spectrum is positive. This map regarded as a densely defined operator in the $L^2$-space of the boundary admits Friedrichs extension. We focus on the spectrum of this operator on covering spaces and total spaces of Riemannian principal bundles over compact manifolds.
\end{abstract}

\section{Introduction}

During the last years, the Steklov spectrum of compact Riemannian manifolds has been extensively studied, and analogues of various results on the Dirichlet and the Neumann spectrum have been established (cf. for instance the survey \cite{GP}, or \cite{FS} and the references therein).
However, the Steklov spectrum of non-compact Riemannian manifolds has not attracted that much attention yet. This is reasonable, since even the definition of Dirichlet-to-Neumann maps is quite more complicated in this case. Indeed, there may exist compactly supported smooth functions on the boundary of such a manifold which do not admit unique harmonic extension even under constraints, such as square-integrability or boundedness, or the normal derivative of the harmonic extension does not satisfy integrability conditions to give rise to an operator in a Hilbert space.

In this paper, we focus on a certain Dirichlet-to-Neumann map on Riemannian manifolds with boundary and bounded geometry, in the sense of \cite{GN,MR1817852}. For a Riemannian manifold $M$ with boundary, we denote by $\nu$ the outward pointing unit normal to the boundary, and by $\lambda_0^D(M)$ the bottom of the Dirichlet spectrum of $M$. The basis of our discussion is that if $\lambda_0^D(M) > 0$, then there exists a natural Dirichlet-to-Neumann map on $M$, as illustrated in the following.

\begin{theorem}\label{unique harmonic ext thm}
Let $M$ be a Riemannian manifold with boundary and bounded geometry, such that $\lambda_0^D(M) > 0$. Then any $f \in C^{\infty}_c(\partial M)$ admits a unique square-integrable harmonic extension $\mathcal{H}f \in C^{\infty}(M)$. Moreover, this extension satisfies $\nu(\mathcal{H}f) \in L^{2}(\partial M)$.
\end{theorem}

The main ingredient in the preceding theorem is that the square-integrable harmonic extension actually belongs to $H^2(M)$. This relies on elliptic estimates on manifolds with boundary and bounded geometry, which were recently proved in \cite{GN}. In view of \Cref{unique harmonic ext thm}, we may consider the Dirichlet-to-Neumann map
\begin{equation*}
	\Lambda \colon C^{\infty}_c(\partial M) \subset L^{2}(\partial M) \to L^{2}(\partial M), \text{ } f \mapsto \nu(\mathcal{H}f).
\end{equation*}
This linear operator admits Friedrichs extension, being densely-defined, symmetric and non-negative definite. The spectrum of this self-adjoint operator is called the \textit{Steklov spectrum} of $M$. It is worth to point out that if $M$ is compact, then this definition coincides with the standard one in the literature. The first part of this paper is devoted to the study of some basic properties of this operator.

In the second part of the paper, we focus on the behavior of the Steklov spectrum under Riemannian coverings. The philosophy of such results is that some properties of the fundamental group of a compact manifold are reflected in the geometry of its universal covering space. A classic result in this direction is due to Brooks \cite{Brooks} asserting that the fundamental group of a closed (that is, compact and without boundary) Riemannian manifold is amenable if and only if the bottom of the spectrum of the Laplacian on its universal covering space is zero. The analogous results for manifolds with boundary involving the Dirichlet and the Neumann spectrum have been recently established in \cite{Mine2}. It is also worth to mention that according to \cite{Mine}, if the fundamental group of the manifold is amenable then its spectrum is contained in the spectrum of its universal covering space.

It is noteworthy that any covering space of a compact manifold with boundary has bounded geometry, and the bottom of its Dirichlet spectrum is positive. Hence, we may define the Steklov spectrum of the covering space as above. In this setting, we prove the analogue of Brooks' result for manifolds with boundary involving the Steklov spectrum.

\begin{theorem}\label{universal covering thm}
Let $M$ be a compact Riemannian manifold with boundary and denote by $\tilde{M}$ its universal covering space. Then the following are equivalent:
\begin{enumerate}[topsep=0pt,itemsep=-1pt,partopsep=1ex,parsep=0.5ex,leftmargin=*, label=(\roman*), align=left, labelsep=0em]
\item $\pi_1(M)$ is amenable,
\item the Steklov spectra satisfy the inclusion $\sigma(M) \subset \sigma(\tilde{M})$,
\item the bottom of the Steklov spectrum of $\tilde{M}$ is zero.
\end{enumerate}
\end{theorem}

It seems interesting that the topology of the boundary does not play any role in the preceding theorem, taking into account that we consider operators acting on functions defined on the boundary (which is a significant difference from the Dirichlet and the Neumann analogues of Brooks' result). For instance, the fundamental group of any boundary component may be non-amenable, while the fundamental group of the manifold is amenable. Furthermore, the fundamental group of any boundary component may be amenable (or even trivial), while the fundamental group of the manifold is non-amenable.

In the third part of the paper, we consider the Steklov spectrum of total spaces of Riemannian principal bundles over compact manifolds. The notion of Riemannian submersion has been introduced in 1960s as a tool to describe the geometry of a manifold in terms of simpler components (the base space and the fibers). Hence, it is natural to express the spectrum of the total space in terms of the geometry and the spectrum of the base space and the fibers. There are various results in this direction involving compact (cf. for instance the survey \cite{MR2963622}) or non-compact manifolds (see for example \cite{Bessa,MR3787357,Mine4,Mine3}). Our discussion is motivated by the recent paper \cite{Mine3}, which focuses on non-compact total spaces of Riemannian principal bundles (see also \cite{MR2348279} for the compact case).

To set the stage, let $G$ be a connected Lie group acting freely, smoothly and properly via isometries on a Riemannian manifold $M_2$. Then the quotient $M_1 = M_2/G$ is a Riemannian manifold and the projection $p \colon M_2 \to M_1$ is a Riemannian submersion. We then say that $p$ \textit{arises from the action of} $G$. Similarly to the case of Riemannian coverings, we are interested in how properties of $G$ are reflected in the spectrum of $M_2$. As indicated in \cite{Mine3}, it is natural to compare the spectrum of the Laplacian on $M_2$ with the spectrum of the Schr\"odinger operator
\[
S = \Delta + \frac{1}{4} \| p_*H \|^2 - \frac{1}{2} \diver p_*H
\]
on $M_1$, where $H$ is the mean curvature of the fibers. More precisely, according to \cite[Theorem 1.3]{Mine3}, if $M_1$ is compact, then $G$ is unimodular and amenable if and only if $\lambda_0^D(M_2) = \lambda_0^D(S)$. It is noteworthy that if $G$ is unimodular, then $S$ is intertwined with the symmetric diffusion operator
\[
L = \Delta + p_*H
\]
regarded in $L^2_{\sqrt{V}}(M_1)$, where $V$ is a function expressing the volume element of the fiber. If $M_1$ is compact, then we have that $\lambda_0^D(L) > 0$, which yields that any $f \in C^{\infty}(\partial M_1)$ has a unique $L$-harmonic extension $\mathcal{H}_L f \in C^{\infty}(M_1)$ (that is, $L (\mathcal{H}_Lf) = 0$). Hence, we may consider the Dirichlet-to-Neumann map
\[
\Lambda_{L} \colon C^{\infty}(\partial M_1) \subset L^2_{\sqrt{V}}(\partial M_1) \to L^2_{\sqrt{V}}(\partial M_1), \text{ } f \mapsto \nu(\mathcal{H}_L f).
\]
We denote the spectrum of the Friedrichs extension of this operator by $\sigma_L(M_1)$.

Returning to our discussion on the Steklov spectrum, we begin by pointing out that if $p \colon M_2 \to M_1$ is a Riemannian submersion arising from the action of a connected Lie group $G$, where $M_1$ is compact, then $M_2$ has bounded geometry and $\lambda_0^D(M_2) > 0$ (the latter one is a consequence of the recent \cite[Theorem 1.1]{Mine3}). Therefore, we may define the Steklov spectrum of $M_2$ as above. In this setting, we establish the following analogue of \cite[Theorem 1.3]{Mine3}.

\begin{theorem}\label{Submersion thm}
	Let $p \colon M_2 \to M_1$ be a Riemannian submersion arising from the action of a connected Lie group $G$, where $M_1$ is compact with boundary. Then the following are equivalent:
	\begin{enumerate}[topsep=0pt,itemsep=-1pt,partopsep=1ex,parsep=0.5ex,leftmargin=*, label=(\roman*), align=left, labelsep=0em]
		\item $G$ is unimodular and amenable,
		\item the Steklov spectra satisfy the inclusion $\sigma_{L}(M_1) \subset \sigma(M_2)$,
		\item the bottom of the Steklov spectrum of $M_2$ is zero.
	\end{enumerate}
\end{theorem}

The paper is organized as follows: In Section \ref{Preliminaries}, we give some preliminaries involving functional analysis, the spectrum of Laplace type operators, Sobolev spaces on manifolds with bounded geometry, amenable coverings, and Lie groups. Section \ref{Steklov spectrum} is devoted to the proof of \Cref{unique harmonic ext thm} and the discussion of some properties of the Steklov spectrum. In Section \ref{Coverings}, we focus on Riemannian coverings and establish an extension of \Cref{universal covering thm}. In Section \ref{submersions}, we study Riemannian submersions and give the proof of \Cref{Submersion thm}.

\medskip

\textbf{Acknowledgements.} I would like to thank Werner Ballmann for his helpful comments and remarks. I am also grateful to the Max Planck Institute for Mathematics in Bonn for its support and hospitality.

\section{Preliminaries}\label{Preliminaries}

We begin by recalling some basic facts from functional analysis, which may be found for instance in \cite{MR1361167,MR1335452}. Throughout this section, let $\mathsf{H}$ be a separable Hilbert space over $\R$.

The \textit{spectrum} of a self-adjoint operator $T \colon \mathcal{D}(T) \subset \mathsf{H} \to \mathsf{H}$ is defined as
\[
\sigma(T) = \{ \lambda \in \R : T - \lambda : \mathcal{D}(T) \subset \mathsf{H} \to \mathsf{H} \text{ is not bijective} \}.
\]
In general, the spectrum of a self-adjoint operator does not consist only of eigenvalues, but consists of approximate eigenvalues, as the following proposition indicates.

\begin{proposition}\label{approximate eigenvalues}
Let $T \colon \mathcal{D}(T) \subset \mathsf{H} \to \mathsf{H}$ be a self-adjoint operator and consider $\lambda \in \R$. Then $\lambda \in \sigma(T)$ if and only if there exists $(v_n)_{n \in \N} \subset \mathcal{D}(T)$ such that $\| v_n \| = 1$ and $(T-\lambda)v_n \rightarrow 0$ in $\mathsf{H}$.
\end{proposition}

The infimum of $\sigma(T)$ is called the \textit{bottom} of the spectrum of $T$ and is denoted by $\lambda_0(T)$. A particularly useful expression for $\lambda_0(T)$ is provided by Rayleigh's theorem, which asserts that
\[
\lambda_0(T) = \inf_{v} \frac{\< Tv,v \>}{\| v \|^2},
\]
where the infimum is taken over all non-zero $v \in \mathcal{D}(T)$.

Consider now a densely-defined, symmetric linear operator $T \colon \mathcal{D}(T) \subset \mathsf{H} \to \mathsf{H}$. We say that $T$ is \textit{bounded from below} if there exists $c\in\R$ such that $\< Tv,v \> \geq c \| v \|^2$ for any $v \in \mathcal{D}(T)$. Fix such a $c$ and define the inner product
\[
\< v,w \>_{T} = \< Tv,w \> + (1-c) \< v,w \>
\]
on $\mathcal{D}(T)$. Denoting by $\mathsf{H}_T$ the completion of $\mathcal{D}(T)$ with respect to this inner product, it is easy to see that the inclusion $\mathcal{D}(T) \xhookrightarrow{} \mathsf{H}$ extends to an injective map $\mathsf{H}_T \to \mathsf{H}$. Using this map, we regard $\mathsf{H}_{T}$ as a subspace of $\mathsf{H}$. The \textit{Friedrichs extension} $T^F$ of $T$ is the restriction of the adjoint $T^*$ to the space $\mathcal{D}(T^F) = \mathsf{H}_T \cap \mathcal{D}(T^*)$. The Friedrichs extension $T^F$ is a self-adjoint extension of $T$, and Rayleigh's theorem implies the following expression for the bottom of its spectrum.

\begin{proposition}\label{bottom of Friedrichs extension}
The bottom of the spectrum of $T^F$ is given by
\[
\lambda_0(T^F) = \inf_{v} \frac{\< Tv,v  \>}{\| v \|^2},
\]
where the infimum is taken over all non-zero $v \in \mathcal{D}(T)$.
\end{proposition}

\subsection{Laplace type operators}

Throughout this paper manifolds are assumed to be connected. However, their boundaries may be non-connected. Moreover, the term "manifold with boundary" refers to a manifold with non-empty, smooth boundary. 

Let $M$ be a Riemannian manifold with possibly empty boundary. A \textit{symmetric Laplace type operator} $L$ on $M$ is an operator of the form $L = \Delta - 2 \grad \ln \vf + V$, where $\Delta$ is the Laplacian, $\vf, V \in C^{\infty}(M)$ and $\vf > 0$. In the case where $\vf = 1$, $L$ is a \textit{Schr\"odinger operator}, while $L$ is called a \textit{diffusion operator} if $V = 0$. Denote by $L^2_{\vf}(M)$ the $L^2$-space of $M$ with respect to the measure $\vf^2 d\vol$, where $d \vol$ is the measure induced by the Riemannian metric.
It is worth to point out that the isometric isomorphism $m_{\vf} \colon L^2_{\vf}(M) \to L^{2}(M)$ defined by $m_{\vf}f=\vf f$, intertwines $L$ with the Schr\"odinger operator
\[
S = m_{\vf} \circ L \circ m_{\vf}^{-1} = \Delta + V - \frac{\Delta \vf}{\vf}.
\]
Given a non-zero $f \in C^{\infty}_c(M)$, set
\[
\mathcal{R}_L(f) = \frac{\int_M ( \| \grad f \|^2 + V f^2) \vf^2}{\int_M f^2 \vf^2}.
\]
In the case of the Laplacian (that is, $\vf = 1$ and $V =0$), we denote this quantity by $\mathcal{R}(f)$.

If $M$ does not have boundary, then
\[
L \colon C^{\infty}_c(M) \subset L^2_{\vf}(M) \to L^2_{\vf}(M)
\]
is densely-defined and symmetric. If it is bounded from below, we denote by $\lambda_0(L)$ the bottom of the spectrum of its Friedrichs extension. From \Cref{bottom of Friedrichs extension} and the divergence formula, we obtain the following expression for $\lambda_0(L)$.

\begin{proposition}\label{bottom no boundary}
The bottom of the spectrum of $L$ is given by
\[
\lambda_0(L) = \inf_{f} \frac{\< L f , f \>_{L^2_{\vf}(M)}}{\| f \|^2_{L^2_{\vf}(M)}} = \inf_{f} \mathcal{R}_L(f),
\]
where the infimum is taken over all non-zero $f \in C^{\infty}_c(M)$.
\end{proposition}

For the rest of this subsection, suppose that $M$ has non-empty boundary. We begin our discussion with the Dirichlet spectrum of $L$. The operator
\begin{equation}\label{Dirichlet}
L  \colon \{ f \in C_c^{\infty}(M) : f = 0 \text{ on } \partial M \} \subset L^{2}_{\vf}(M) \to L^{2}_{\vf}(M)
\end{equation}
is densely-defined and symmetric. If it is bounded from below, we denote by $L^D$ its Friedrichs extension. It is noteworthy that if $M$ is complete, then this operator is essentially self-adjoint (cf. for instance \cite[Theorem A.24]{BP}); that is, $L^D$ is the closure of this operator and actually coincides with the adjoint of this operator. The spectrum of $L^D$ is called the \textit{Dirichlet spectrum} of $L$. The following expression for the bottom $\lambda_0^D(L)$ of the Dirichlet spectrum is an immediate consequence of \Cref{bottom of Friedrichs extension} and the divergence formula.

\begin{proposition}\label{bottom Dirichlet}
The bottom of the Dirichlet spectrum of $L$ is given by
\[
\lambda_0^D(L) = \inf_{f} \frac{\< L f , f \>_{L^2_{\vf}(M)}}{\| f \|^2_{L^2_{\vf}(M)}} = \inf_{f} \mathcal{R}_L(f),
\]
where the infimum is taken over all non-zero $f \in C^{\infty}_c(M)$ with $f = 0 $ on $\partial M$.
\end{proposition}

It is not difficult to verify that the bottom of the Dirichlet spectrum of $L$ coincides with the bottom of the spectrum of $L$ considered on the interior of $M$.

We are also interested in the \textit{Neumann spectrum} of $L$, which is the spectrum of the Friedrichs extension $L^N$ of the operator
\[
L \colon \{ f \in C_c^{\infty}(M) : \nu(f) = 0 \text{ on } \partial M \} \subset L^{2}_{\vf}(M) \to L^{2}_{\vf}(M),
\]
in the case where this operator is bounded from below, where $\nu$ stands for the outward pointing unit normal to $\partial M$. The following expression for the bottom $\lambda_0^N(L)$ of the Neumann spectrum may be found for instance in \cite[Proposition 3.2]{Mine2} in the case of Schr\"odinger operators. This readily extends the to symmetric Laplace type operators, by passing to the corresponding Schr\"odinger operator as described in the beginning of this subsection.

\begin{proposition}\label{bottom Neumann}
The bottom of the Neumann spectrum of $L$ is given by
\[
\lambda_0^N(L) = \inf_{f} \mathcal{R}_L(f),
\]
where the infimum is taken over all non-zero $f\in C^{\infty}_c(M)$.
\end{proposition}

It should be emphasized that the test functions in the preceding proposition are not required to satisfy any boundary condition.

To simplify our notation, in the case of the Laplacian, we set $\lambda_0(M) = \lambda_0(\Delta)$, $\lambda_0^D(M) = \lambda_0^D(\Delta)$ and $\lambda_0^N(M) = \lambda_0^N(\Delta)$.

\subsection{Sobolev spaces on manifolds with bounded geometry}

Let $M$ be a Riemannian manifold with boundary and $\nu$ the outward pointing unit normal to $\partial M$. Denote by $R$ the curvature tensor of $M$ and by $\alpha$ the second fundamental form of $\partial M$. The following definition may be found in \cite{GN,MR1817852}\footnote{The definition in \cite{GN} is equivalent to the one in \cite{MR1817852}, as pointed out in \cite[p. 12]{GN}.}.
\begin{definition}\label{manifold with bounded geometry}
We say that $M$ has \textit{bounded geometry} if the following hold:
\begin{enumerate}[topsep=0pt,itemsep=-1pt,partopsep=1ex,parsep=0.5ex,leftmargin=*, label=(\roman*), align=left, labelsep=0em]
	\item there exists $r > 0$ such that
	\[
	\exp \colon \partial M \times [0,r) \to M, \text{ } (x,t) \mapsto \exp_x(-t\nu)
	\]
	is a diffeomorphism onto its image,
	\item the injectivity radius of $\partial M$ (as a manifold endowed with the induced Riemannian metric) is positive,
	\item there is $r_0 > 0$ such that for any $x \in M \smallsetminus B(\partial M, r_0)$, the restriction of $\exp_x$ to $B(0,r_0) \subset T_x M$ is a diffeomorphism onto its image,
	\item for any $k \in \N \cup\{0\}$ there exists $C_k > 0$ such that $\| \nabla^k R \| \leq C_k$ and $\| \nabla^k \alpha \| \leq C_k$.
\end{enumerate}
\end{definition}

We begin our discussion on such manifolds with the following observation.

\begin{lemma}\label{extension}
Let $M$ be a Riemannian manifold with boundary and bounded geometry. Then the outward pointing unit normal $\nu$ to $\partial M$ can be extended to a bounded smooth vector field $N$ on $M$ with $\nabla N$ bounded.
\end{lemma}

\begin{proof}
Denoting by $d_{\partial M}$ the distance to $\partial M$, using the diffeomorphism from (i), extend $\nu$ to the vector field
\[
V = \exp_*(-\frac{\partial}{\partial t}) = - \grad d_{\partial M}
\]
defined in $B(\partial M, r)$. It follows from standard comparison theorems (cf. for instance \cite{Eschenburg-Heintze}) that there exists $\delta > 0$ such that $\nabla V$ is bounded in $B(\partial M, \delta)$, keeping in mind that the sectional curvature of $M$ is bounded and so are the principal curvatures of $\partial M$. Consider now a smooth $\chi \colon [0 +\infty) \to \R$ with $\chi(x) = 1$ for $x \leq \delta/4$ and $\chi(x) = 0$ for $x \geq \delta/2$. It is immediate to verify that $N = (\chi \circ d_{\partial M}) V$ is a bounded smooth vector field on $M$ with $\nabla N$ bounded, which coincides with $\nu$ on $\partial M$. \qed
\end{proof}\medskip

The next theorem is essentially a special version of \cite[Theorem 1.1]{GN}, where we point out a difference in the form of the elliptic estimates, in the case where the bottom of the Dirichlet spectrum of the Laplacian is positive.

\begin{theorem}\label{Elliptic estimates}
Let $M$ be a Riemannian manifold with boundary and bounded geometry, such that $\lambda_0^D(M) > 0$. Then any $f \in \mathcal{D}(\Delta^D)$ belongs to $H^{2}(M)$. More precisely, there exists $C>0$ such that
\[
\| f \|_{H^{2}(M)} \leq C \| \Delta^D f \|_{L^{2}(M)}
\]
for any $f \in \mathcal{D}(\Delta^D)$.
\end{theorem}

\begin{proof}
Bearing in mind that the Laplacian on $M$ regarded as in (\ref{Dirichlet}) is essentially self-adjoint, we readily see that for any $f \in \mathcal{D}(\Delta^D)$ there exists $(f_n)_{n \in \N} \subset C^{\infty}_{c}(M)$ with $f_n = 0$ on $\partial M$ such that $f_n \rightarrow f$ and $\Delta f_n \rightarrow \Delta^D f$ in $L^2(M)$.
Therefore, it suffices to establish the asserted estimate for $f \in C^{\infty}_{c}(M)$ with $f=0$ on $\partial M$. Indeed, this gives that $(f_{n})_{n \in \N}$ is Cauchy in $H^2(M)$, and thus, converges to $f$ in $H^{2}(M)$.

It follows from \cite[Theorem 1.1]{GN} that there exists $C>0$ such that
\[
\| f \|_{H^2(M)} \leq C( \| \Delta f \|_{L^{2}(M)} + \| f \|_{H^{1}(M)} )
\]
for any $f \in C^{\infty}_{c}(M)$ with $f=0$ on $\partial M$. If, in addition, $f$ is not identically zero, we readily see from \Cref{bottom Dirichlet} that
\[
\| \Delta f \|_{L^{2}(M)} \geq \frac{\< \Delta f , f \>_{L^2(M)}}{\| f \|_{L^{2}(M)}} = \mathcal{R}(f) \| f \|_{L^{2}(M)} \geq \lambda_0^D(M) \| f \|_{L^2(M)}.
\]
Since $\lambda_0^D(M) > 0$, 
this yields that
\[
\int_{M} \| \grad f\|^2 = \< \Delta f , f \>_{L^2(M)} \leq \| \Delta f \|_{L^{2}(M)}  \| f \|_{L^2(M)} \leq \lambda_0^D(M)^{-1} \| \Delta f \|_{L^{2}(M)}^2.
\]
The proof is completed by combining the above inequalities. \qed
\end{proof}\medskip

We will also exploit the following trace theorem, which is a special version of \cite[Theorem 3.15]{GN}.

\begin{theorem}\label{Trace thm}
Let $M$ be a Riemannian manifold with boundary and bounded geometry. Then the restriction to the boundary $\res \colon C^{\infty}_c(M) \to C^{\infty}_c(\partial M)$ extends to a continuous $\res \colon H^{1}(M) \to L^2(\partial M)$. 
\end{theorem}

\subsection{Amenable coverings}

Consider a right action of a finitely generated discrete group $\Gamma$ on a countable set $X$. We say that this action is \textit{amenable} if there exists an invariant mean on $\ell^\infty(X)$; that is, a linear functional $\mu \colon \ell^\infty(X) \to \R$ such that:
\begin{enumerate}[topsep=0pt,itemsep=-1pt,partopsep=1ex,parsep=0.5ex,leftmargin=*, label=(\roman*), align=left, labelsep=0em]
\item $\inf f \leq \mu (f) \leq \sup f$,
\item $\mu(f \circ r_g) = \mu(f)$, where $r_g(x) = xg$ for any $x \in X$,
\end{enumerate}
for any $f \in \ell^\infty(X)$ and $g \in \Gamma$.

The group $\Gamma$ is called amenable if the right action of $\Gamma$ on itself is amenable. Standard examples of amenable groups are solvable groups and groups of subexponential growth. It is not difficult to verify that if $\Gamma$ is amenable, then any right action of $\Gamma$ is amenable.
A particularly useful characterization of amenability is the following proposition due to F{\o}lner.

\begin{proposition}\label{Folner}
Let $\Gamma$ be a finitely generated generated group and fix a finite, symmetric generating set $G$. Then the right action of $\Gamma$ on a countable set $X$ is amenable if and only if for any $\ve > 0$ there exists a non-empty finite subset $P$ of $X$ such that
\[
| Pg \smallsetminus P| < \ve |P|
\]
for any $g \in G$.
\end{proposition}

Let $p \colon M_2 \to M_1$ be a Riemannian covering and choose $x \in M_{1}$ as a base point for $\pi_1(M_1)$. Given $y \in p^{-1}(x)$ and $g \in \pi_1(M_1)$, consider a representative loop $\gamma$ of $g$ based at $x$. Lift $\gamma$ to a path $\tilde{\gamma}$ starting at $y$ and denote its endpoint by $yg$. In this way, we obtain a right action of $\pi_1(M_1)$ to $p^{-1}(x)$, which is called the \textit{monodromy action}. We say that the covering is \textit{amenable} if the monodromy action is amenable. It is worth to mention that a normal Riemannian covering is amenable if and only if its deck transformation group is amenable. In particular, the universal covering of a manifold is amenable if and only if its fundamental group is amenable.

If $M_1$ is compact with boundary, then $\pi_1(M_1)$ is finitely generated. More specifically, the finite and symmetric set
\[
S_r = \{ g \in \pi_1(M_1) : g \text{ has a representative loop of length less than } r \}
\]
generates $\pi_1(M_1)$ for $r > 0$ sufficiently large. The next elementary lemma provides a description of the monodromy action in terms of this set.

\begin{lemma}\label{action of Gr}
Given $y_1 , y_2 \in p^{-1}(x)$, there exists $g \in S_r$ such that $y_2 = y_1 g$ if and only if $d(y_1,y_2) < r$.
\end{lemma}

\begin{proof}
Suppose that $y_2 = y_1 g$ for some $g \in S_r$. Then there exists a representative loop $\gamma$ of $g$ based at $x$ of length less than $r$. Since the endpoint of its lift starting at $y_1$ is $y_2$, it is clear that $d(y_1,y_2) < r$. Conversely, if $d(y_1,y_2) < r$, consider a curve $c$ from $y_1$ to $y_2$ of length less than $r$. Denoting by $g$ the class of $p \circ c$ in $\pi_1(M_1)$, we readily see that $y_2 = y_1 g$ and $g \in S_r$, $p \circ c$ having length less than $r$. \qed \medskip
\end{proof}

Amenability of a covering is intertwined with the preservation of the bottom of the spectrum, as Brooks' result illustrates. In the sequel, we will exploit the corresponding result involving the bottom of the Neumann spectrum.

\begin{theorem}[{{\cite[Theorem 1.1]{Mine2}}}]\label{Brooks Neumann}
Let $p \colon M_2 \to M_1$ be a Riemannian covering, where $M_1$ is compact with boundary. Then $p$ is amenable if and only if $\lambda_0^N(M_2) = 0$.
\end{theorem}

\subsection{Lie groups}

A connected Lie group $G$ is called \textit{amenable} if there exists a left-invariant mean on $L^{\infty}(G)$; that is a linear functional $\mu \colon L^{\infty}(G) \to \R$ such that
\begin{enumerate}[topsep=0pt,itemsep=-1pt,partopsep=1ex,parsep=0.5ex,leftmargin=*, label=(\roman*), align=left, labelsep=0em]
	\item $\essinf f \leq \mu(f) \leq \esssup f $,
	\item $\mu(f \circ L_x) = \mu(f)$,
\end{enumerate}
for any $f \in L^{\infty}(G)$ and $x \in G$, where $L_x \colon G \to G$ stands for multiplication from the left with $x \in G$. Here, we consider $L^\infty(G)$ with respect to the Haar measure of $G$, which is just a constant multiple of the volume element of $G$ induced by a left-invariant metric. For more details, see \cite{MR0251549}. It is well-known that a connected Lie group $G$ is amenable if and only if it is a compact extension of a solvable group (cf. for example \cite[Lemma 2.2]{MR454886}).

A Lie group $G$ is called \textit{unimodular} if its Haar measure is right-invariant. It is noteworthy that a connected Lie group $G$ is unimodular if and only if $\tr (\ad X) = 0$ for any $X$ in the Lie algebra of $G$ (cf. for instance \cite[Proposition 1.2]{Hoke}). Standard examples of unimodular and amenable Lie groups are connected, nilpotent Lie groups.

Even thought the above properties are group theoretic, they are reflected in the spectrum of the Laplacian. The following characterization has been established for simply connected Lie groups in \cite[Theorem 3.8]{Hoke}, and extended to connected Lie groups in \cite[Theorem 2.10]{Mine}.

\begin{theorem}\label{unimodular and amenable}
	A connected Lie group $G$ is unimodular and amenable if and only if $\lambda_0(G) = 0$ for some/any left-invariant metric on $G$.
\end{theorem}

\section{The Steklov spectrum of manifolds with bounded geometry}\label{Steklov spectrum}

In this section, we define the Steklov spectrum of manifolds satisfying the assumptions of \Cref{unique harmonic ext thm}, and discuss some basic properties of it. Throughout this section, we consider a Riemannian manifold $M$ with boundary and bounded geometry such that $\lambda_0^D(M) > 0$, and we denote by $\nu$ the outward pointing unit normal to $\partial M$.

\begin{proposition}\label{existence of harmonic extension}
Any $f \in C^{\infty}_c(\partial M)$ has a unique square-integrable harmonic extension $\mathcal{H}f \in C^{\infty}(M)$. In addition, this extension is written as $\mathcal{H}f = f + h$, where $f \in C^{\infty}_c(M)$ is an extension of $f$ and $h \in \mathcal{D}(\Delta^D)$.
\end{proposition}

\begin{proof}
Let $f \in C^{\infty}_c(\partial M)$ and consider an extension $f \in C^{\infty}_c(M)$ of it. The assumption that $\lambda_0^D(M) > 0$ means that
\[
\Delta^D \colon \mathcal{D}(\Delta^D) \subset L^2(M) \to L^2(M)
\]
is bijective. In particular, there exists $h \in \mathcal{D}(\Delta^D)$ such that $\Delta^D h = - \Delta f$. Bearing in mind that $f \in C^{\infty}(M)$, we derive from elliptic regularity that $h \in C^{\infty}(M)$ with $h = 0$ on $\partial M$, and $\Delta h = - \Delta f$. It is evident that $h + f$ is a square-integrable harmonic extension of $f$.

Let $f_1, f_2 \in C^{\infty}(M)$ be square-integrable harmonic extensions of a given $f \in C^{\infty}_c(M)$. Then $h = f_1 - f_2 \in C^{\infty}(M) \cap L^2(M)$ is harmonic and vanishes on $\partial M$. Given $g \in C^\infty_{c}(M)$ with $g = 0$ on $\partial M$, we compute
\[
\< h , \Delta g \>_{L^2(M)} = \int_{M} \< \grad h, \grad g \> - \int_{\partial M} h \nu(g) = \int_{\partial M} \nu(h) g = 0,
\]
where we used that $h$ is harmonic. This shows that $h \in \mathcal{D}(\Delta^*)$ with $\Delta^* h = 0$. Since $M$ is complete, the Laplacian $\Delta$ regarded as in (\ref{Dirichlet}) is essentially self-adjoint, and hence, $\Delta^* = \Delta^D$. We deduce that $h$ belongs to the kernel of $\Delta^D$ and thus, $h = 0$, due to the fact that $\lambda_0^D(M) > 0$.\qed
\end{proof}

\begin{proposition}\label{approximation}
For any $f \in C^{\infty}_c(\partial M)$ we have that $\mathcal{H}f \in H^2(M)$. Moreover, there exists a sequence $(f_n)_{n \in \mathbb{N}} \subset C^{\infty}_c(M) $, consisting of extensions of $f$, such that $f_n \rightarrow \mathcal{H}f$ in $H^{2}(M)$.
\end{proposition}

\begin{proof}
Given $f \in C^{\infty}_c(\partial M)$, it follows from \Cref{existence of harmonic extension} that $\mathcal{H}f = f + h$, for some extension $f \in C^{\infty}_c(M)$ of $f$ and $h \in \mathcal{D}(\Delta^D)$. Since the Laplacian on $M$ regarded as in (\ref{Dirichlet}) is essentially self-adjoint, we readily see that there exists $(h_{n})_{n \in \N} \subset C^{\infty}_c(M)$ with $h_n = 0$ on $\partial M$, such that $h_n \rightarrow h$ and $\Delta h_n \rightarrow \Delta h$ in $L^2(M)$. Then \Cref{Elliptic estimates} yields that $h_n \rightarrow h$ in $H^2(M)$. Therefore, $f_n = f + h_n \in C^{\infty}_c(M)$ is an extension of $f$ and $f_n \rightarrow \mathcal{H}f$ in $H^2(M)$. \qed
\end{proof}\medskip

\noindent{\emph{Proof of \Cref{unique harmonic ext thm}:}} We know from \Cref{existence of harmonic extension,,approximation} that any $f \in C^{\infty}_c(\partial M)$ admits a unique square-integrable harmonic extension $\mathcal{H}f \in H^2(M)$. In view of \Cref{extension}, we may extend $\nu$ to a bounded smooth vector field $N$ on $M$ with $\nabla N$ bounded. Then we have that $N(\mathcal{H}f) \in H^{1}(M)$, and its restriction $\nu(\mathcal{H}f)$ to the boundary is square-integrable, by virtue of \Cref{Trace thm}. \qed\medskip

It is now clear that the Dirichlet-to-Neumann map
\[
\Lambda \colon C^{\infty}_c(\partial M) \subset L^{2}(\partial M) \to L^{2}(\partial M), \text{ } f \mapsto \nu(\mathcal{H}f)
\]
is well-defined. It is also evident that this linear operator is densely-defined.

\begin{lemma}\label{symmetry}
For any $f,h \in C^{\infty}_c(\partial M)$, we have that
\[
\< \Lambda f , h \>_{L^{2}(\partial M)} = \< f , \Lambda h \>_{L^{2}(\partial M)} = \int_{M} \< \grad (\mathcal{H}f) , \grad (\mathcal{H}h) \>.
\]
\end{lemma}

\begin{proof}
For any extension $f \in C^{\infty}_c(M)$ of $f$, we derive from the divergence formula that
\[
\< f , \Lambda h \>_{L^{2}(\partial M)} = \int_{M} \< \grad f , \grad (\mathcal{H}h) \>.
\]
The proof is completed by \Cref{approximation}, which asserts that there exists a sequence $(f_n)_{n \in \N}$ of such extensions satisfying $f_n \rightarrow \mathcal{H}f$ in $H^{2}(M)$. \qed
\end{proof}\medskip

Hence, aforementioned Dirichlet-to-Neumann map admits Friedrichs extension, being symmetric and bounded from below by zero. The spectrum $\sigma(M)$ of its Friedrichs extension is called the \textit{Steklov spectrum} of $M$, and its bottom is denoted by $\sigma_0(M)$.

\begin{proposition}\label{bottom of Steklov}
The bottom of the Steklov spectrum of $M$ is given by
\[
\sigma_0(M) = \inf_{f} \frac{\int_{M} \| \grad f \|^2}{\int_{\partial M} f^2},
\]
where the infimum is taken over all $f \in C^{\infty}_c(M)$ which are not identically zero on $\partial M$.
\end{proposition}

Before proceeding to the proof of this proposition, we need the following remark.

\begin{lemma}\label{harmonmic extension minimizes energy}
For any extension $f \in C^{\infty}_{c}(M)$ of a function $f \in C^{\infty}_{c}(\partial M)$, we have that
\[
\int_{M} \| \grad (\mathcal{H}f) \|^2 \leq \int_{M} \| \grad f \|^2.
\]
\end{lemma}

\begin{proof}
Since $\mathcal{H}f - f = h \in H_{0}^1(M)$, there exists $(h_{n})_{n \in \N} \subset C^{\infty}_{c}(M)$ with $h_n = 0$ on $\partial M$, such that $h_n \rightarrow h$ in $H^{1}(M)$. It is immediate to verify that
\[
\int_{M} \| \grad f \|^2 = \int_{M} \| \grad (\mathcal{H}f) \|^2 + \int_{M} \| \grad h \|^2 - 2 \int_{M} \< \grad (\mathcal{H}f) , \grad h \>,
\]
while
\[
\int_{M}\< \grad (\mathcal{H}f) , \grad h \> = \lim_{n} \int_{M} \< \grad (\mathcal{H}f) , \grad h_n \> = \lim_{n} \int_{\partial M} \nu(\mathcal{H}f) h_n = 0,
\]
which establishes the asserted inequality. \qed 
\end{proof}\medskip

From the proof, it is clear that the equality in the preceding lemma holds if and only if $M$ is compact and $f = \mathcal{H}f$. \medskip

\noindent{\emph{Proof of \Cref{bottom of Steklov}:}}
We know from \Cref{bottom of Friedrichs extension} and \Cref{symmetry} that
\[
\sigma_0(M) = \inf_{f} \frac{\<  \Lambda f ,f \>_{L^{2}(\partial M)}}{\| f \|_{L^{2}(\partial M)}^2} = \inf_{f} \frac{\int_{M}\| \grad (\mathcal{H}f)  \|^2}{\int_{\partial M} f^2},
\]
where the infimum is taken over all non-zero $f \in C^{\infty}_c(\partial M)$.
This is less or equal to the asserted infimum, by virtue of \Cref{harmonmic extension minimizes energy}. The equality follows from \Cref{approximation}, according to which, the square-integrable harmonic extension can be approximated in $H^{2}(M)$ by compactly supported smooth extensions.\qed\medskip

A straightforward consequence of \Cref{bottom of Steklov} is the following relation between the bottom of the Steklov and the Neumann spectrum.

\begin{theorem}\label{Steklov and Neumann}
If $\sigma_0(M) = 0$, then $\lambda_0^N(M) = 0$.
\end{theorem}

\begin{proof}
We readily see from \Cref{bottom of Steklov} that there exists $(f_{n})_{n \in \N} \subset C^{\infty}_c(M)$ such that $\| f_{n} \|_{L^2(\partial M)} = 1$ and 
\[
\int_{M}\| \grad f_n \|^2 \rightarrow 0.
\]
Assume to the contrary that $\lambda_0^N(M) > 0$. Then \Cref{bottom Neumann} implies that
\[
\| f_{n} \|_{L^{2}(M)}^2 \leq \lambda_0^N(M)^{-1} \int_{M} \| \grad f_n \|^2 \rightarrow 0.
\]
This means that $f_n \rightarrow 0$ in $H^1(M)$ and thus, $f_n \rightarrow 0$ in $L^2(\partial M)$, by virtue of \Cref{Trace thm}, which is a contradiction. \qed
\end{proof}\medskip

Recall that according to \Cref{approximate eigenvalues}, the spectrum of a self-adjoint operator consists of the approximate eigenvalues of the operator. In general, given $f \in C^{\infty}_c(\partial M)$ and $\lambda \in \R$, it may be quite complicated to estimate the quantity $\| \Lambda f - \lambda f \|_{L^{2}(\partial M)}$. The next observation allows us to substitute the harmonic extension with any compactly supported smooth extension, and the error term is controlled in terms of the Laplacian of the chosen extension.

\begin{proposition}\label{approximation of test functions}
There exists $C>0$ such that
	\[
\| \Lambda f - \lambda f \|_{L^{2}(\partial M)} \leq \| \nu(f) - \lambda f \|_{L^{2}(\partial M)} + C \| \Delta f \|_{L^{2}(M)}
	\]
for any extension $f \in C^{\infty}_c(M)$ of any $f \in C^{\infty}_c(\partial M)$.
\end{proposition}

\begin{proof}
Let $f \in C^{\infty}_c(M)$ be an extension of a given $f \in C^{\infty}_c(\partial M)$. Extending $\nu$ to a bounded smooth vector field $N$ on $M$ with $\nabla N$ bounded (as in \Cref{extension}), we obtain from \Cref{Trace thm} that there exists $C_1 > 0$ such that
\[
\| \nu(f) - \nu(\mathcal{H}f) \|_{L^{2}(\partial M)} \leq C_1 \| N(f) - N(\mathcal{H}f) \|_{H^{1}(M)} \leq C_1 C_{2} \| f - \mathcal{H}f \|_{H^{2}(M)},
\]
where $C_2$ is a constant depending on $N$. In view of \Cref{existence of harmonic extension}, it is apparent that $f - \mathcal{H}F \in \mathcal{D}(\Delta^D)$. Therefore, we derive from \Cref{Elliptic estimates} that there exists $C_3 > 0$ such that
\[
 \| f - \mathcal{H}f \|_{H^{2}(M)} \leq C_3 \| \Delta (f - \mathcal{H}f) \|_{L^{2}(M)} = C_3 \| \Delta f  \|_{L^{2}(M)}.
\]
We conclude that
\begin{eqnarray}
\| \Lambda f - \lambda f \|_{L^{2}(\partial M)} &\leq& \| \nu(f) - \lambda f \|_{L^{2}(\partial M)} + \| \nu(f) - \nu(\mathcal{H}f) \|_{L^2(\partial M)} \nonumber\\
&\leq& \| \nu(f) - \lambda f \|_{L^{2}(\partial M)} +  C_1C_2C_3 \| \Delta f \|_{L^{2}(M)},\nonumber
\end{eqnarray}
as we wished. \qed
\end{proof}

\section{Steklov spectrum under Riemannian coverings}\label{Coverings}

Throughout this section, we consider a Riemannian covering $p \colon M_2 \to M_1$ where $M_{1}$ is compact with boundary. It is not difficult to verify that $M_2$ has bounded geometry. Indeed, it is easily checked that properties (ii)-(iv) of \Cref{manifold with bounded geometry} are satisfied, after noticing that $B(\partial M_2 , r) = p^{-1}(B(\partial M_1 , r))$ for any $r>0$, while the validity of (i) is explained for instance in \cite[Lemma 4.2]{Mine2}. It is also important for our discussion that $\lambda_0^D(M_2) > 0$. This follows from \cite[Theorem 1.3]{BMP1}, keeping in mind that $\lambda_0^D(M_2)$ coincides with the bottom of the spectrum of the Laplacian on the interior of $M_2$. Hence, we may define the Steklov spectrum of $M_2$ as in the previous section. The aim of this section is to establish the following extension of \Cref{universal covering thm}. 

\begin{theorem}\label{covering thm}
Let $p \colon M_2 \to M_1$ be a Riemannian covering, where $M_1$ is compact with boundary. Then the following are equivalent:
\begin{enumerate}[topsep=0pt,itemsep=-1pt,partopsep=1ex,parsep=0.5ex,leftmargin=*, label=(\roman*), align=left, labelsep=0em]
\item $p$ is amenable,
\item the Steklov spectra satisfy the inclusion $\sigma(M_1) \subset \sigma(M_2)$,
\item the bottom of the Steklov spectrum is preserved; that is, $\sigma_0(M_2) = 0$.
\end{enumerate}
\end{theorem}

Let $\tilde{M}$ be the universal covering space of $M_1$, consider the Riemannian coverings $p_i \colon \tilde{M} \to M_i$, and denote by $\Gamma_i$ the deck transformation group of $p_i$, $i=1,2$. It should be noticed that $p \circ p_2 = p_1$.

We begin by introducing the fundamental domains that will be used in the sequel. It is worth to point out that the Dirichlet fundamental domains used in \cite{BMP1,Mine} cannot be exploited in our setting, since we have to deal with integrals over the boundary.

Fix a finite, smooth triangulation of $M_1$ that induces a triangulation of $\partial M_1$, and for each full-dimensional simplex choose a lift on $\tilde{M}$, so that the union $F$ of their images is connected. The set $F$ is called a \textit{finite sided fundamental domain} of the covering $p_1$.
We readily see that $\Vol(\partial F) = \Area (\partial \tilde{M} \cap \partial F) = 0$ and the translates $gF$ of $F$ with $g \in \Gamma_1$ cover $\tilde{M}$. Here, $\partial F$ stands for the boundary of $F$ as a subset of $\tilde{M}$, and similarly below. It is also immediate to verify that
\[
\int_{F} (f\circ p_1) = \int_{M_1} f \text{ and } \int_{\partial \tilde{M} \cap F} (f \circ p_1) = \int_{\partial M_1} f
\]
for any $f \in C^{\infty}(M_1)$.

Choose $\tilde{x} \in F^\circ$ and set $x = p_1(\tilde{x})$. Given $y \in p^{-1}(x)$, there exists $g \in \Gamma_1$ such that $g\tilde{x} \in p_2^{-1}(y)$. Set $F_y = p_2(gF)$ and observe that it does not depend on the choice of $g$, since if $g^{\prime}\tilde{x} \in p_2^{-1}(y)$ for some $g^\prime \in \Gamma_1$, then $g^\prime g^{-1} \in \Gamma_2$. It is evident that the domains $F_y$ with $y \in p^{-1}(x)$ cover $M_2$ and $\diam(F_y) \leq \diam(F)$ for any $y \in p^{-1}(x)$. It is not hard to check that $\partial F_y \subset p_2(g \partial F)$, which implies that $\Vol(\partial F_y) = \Area(\partial M_2 \cap \partial F_y) = 0$.
From the fact that $p_1 \colon gF \to M_1$ and $p_1 \colon \partial \tilde{M} \cap \partial F \to \partial M_1$ are isometries up to sets of measure zero, we readily see that so are the restrictions $p \colon F_y \to M_1$ and $p \colon \partial M_2 \cap F  \to  \partial M_1$ for any $y \in p^{-1}(x)$. This yields that
\begin{equation}\label{integration covering}
\int_{F_y} (f\circ p) = \int_{M_1} f \text{ and } \int_{\partial M_2 \cap F_y} (f \circ p) = \int_{\partial M_1} f
\end{equation}
for any $f \in C^{\infty}(M_1)$ and $y \in p^{-1}(x)$. For any $y,z \in p^{-1}(x)$ with $y \neq z$, using that $F_y \cap F_z \subset \partial F_y$, we derive that
\begin{equation}\label{intersections covering}
\Vol(F_y \cap  F_z) = \Area(\partial M_2 \cap F_y \cap  F_z) = 0.
\end{equation}

We now construct a partition of unity on $M_{2}$, which will be used to obtain cut-off functions. To this end, we will exploit the following.

\begin{lemma}\label{cardinality of intersection}
Given $r>0$, there exists $k(r) \in \N$ such that for any $z \in M_2$ the open ball $B(z,1)$ intersects at most $k(r)$ of the closed balls $C(y,r)$ with $y \in p^{-1}(x)$.
\end{lemma}

\begin{proof}
Suppose that $B(z,1)$ intersects the closed balls $C(y_i,r)$ with $y_i \in p^{-1}(x)$ pairwise different and let $\gamma_i$ be a minimizing geodesic from $z$ to $y_i$, $i=1,\dots,k$. Then the concatenations $(p \circ \gamma_i) * (p \circ \gamma_1^{-1})$ are pairwise non-homotopic loops based at $x$ of length less than $2r+2$. We conclude that $k \leq |S_{2r+2}|$.\qed
\end{proof}\medskip

Consider $r>0$ such that $F \subset B(\tilde{x},r)$ and $S_r$ generates $\pi_1(M_1)$. Fix a non-negative $\psi \in C^{\infty}_{c}(\tilde{M})$ with $\psi = 1$ in $B(\tilde{x},r)$ and $\supp \psi \subset B(\tilde{x},r+1)$. Given $y \in p^{-1}(x)$, choose $g \in \Gamma_1$ such that $g\tilde{x} \in p_2^{-1}(y)$ and set
\begin{equation}\label{pushdown}
\psi_{y}(z) = \sum_{w \in p_{2}^{-1}(z)} (\psi \circ g^{-1})(w) = \sum_{g^\prime \in \Gamma_2} (\psi \circ g^{-1} \circ g^\prime) (w_0),
\end{equation}
for some fixed $w_0 \in p_2^{-1}(z)$. Since $\psi$ is compactly supported, we readily see that $\psi_y$ is well-defined, smooth, non-negative, $\psi_y \geq 1$ in $B(y,r)$ and $\supp \psi_y \subset B(y,r+1)$. It follows from \Cref{cardinality of intersection} (applied to the universal covering $p_1$) that there exists $\tilde{k} \in \N$ such that locally at most $\tilde{k}$ terms in the right-hand side of (\ref{pushdown}) are non-zero. Thus, we deduce that there exists $C_1 >0 $ such that
\begin{equation*}
\|\grad \psi_y \| \leq C_1  \text{ and } |\Delta \psi_y | \leq C_1
\end{equation*}
for any $y \in p^{-1}(x)$.

According to \Cref{cardinality of intersection}, there exists $k \in \N$ such that at most $k$ of the supports of $\psi_y$ with $y \in p^{-1}(x)$ intersect any ball of radius one in $M_{2}$. Therefore, the function
\[
\psi = \sum_{y \in p^{-1}(x)} \psi_y
\]
is well-defined, smooth, and greater or equal to one. Moreover, we obtain that $\psi$, $\grad \psi$ and $\Delta \psi$ are bounded. 

Consider now the smooth partition of unity on $M_2$ consisting of the functions
\[
\vf_y = \frac{\psi_y}{\psi},
\]
with $y \in p^{-1}(x)$. It is clear that $\supp \vf_y \subset B(y,r+1)$, $\vf_y > 0$ in $B(y,r)$, while $\grad \vf_y$ and $\Delta \vf_y$ are uniformly bounded for all $y \in p^{-1}(x)$. Given a finite subset $P$ of $p^{-1}(x)$, define the function $\chi_P \in C^{\infty}_c(M_2)$ by
\[
\chi_P = \sum_{y \in P} \vf_y
\]
and the sets
\begin{eqnarray}
Q_+ &=& \{ y \in p^{-1}(x) : \chi = 1 \text{ in } F_y \}, \nonumber \\
Q_- &=& \{ y \in p^{-1}(x) : 0 < \chi(z) < 1 \text{ for some } z \in F_y \}. \nonumber
\end{eqnarray}
Since at most $k$ of the supports of $\vf_y$ interest any ball of radius one in $M_2$, it is easy to see that exists $C_2> 0$ such that
\begin{equation}\label{uniform estimates}
\| \grad \chi_P \| \leq C_2 \text{ and } | \Delta \chi_P | \leq C_2 
\end{equation}
for any finite subset $P$ of $p^{-1}(x)$. Furthermore, the sets $Q_+$ and $Q_-$ are finite, $\chi_P$ being compactly supported

The following proposition illustrates how amenability of the covering is related to our construction.

\begin{proposition}\label{amenability}
If $p \colon M_2 \to M_1$ is amenable, then for any $\ve>0$ there exists a finite subset $P$ of $p^{-1}(x)$ such that $|Q_-| < \varepsilon |Q_+|$.
\end{proposition}

\begin{proof}
Set $d = 2(r + 1 + \diam (F))$. We know from \Cref{Folner} that for any $\ve > 0$ there exists a non-empty, finite subset $P$ of $p^{-1}(x)$ such that
\[
|Pg \smallsetminus P| < \ve |P|
\]
for any $g \in S_{d}$. Consider the corresponding function $\chi_P$ and let $y \in P$ such that $yg \in P$ for any $g \in S_{d}$. Let $z \in F_{y}$ and $y^\prime \in p^{-1}(x)$ such that $z \in \supp \vf_{y^\prime}$. Then $d(z,y^\prime) < r+1$ and thus, $d(y,y^\prime) < d/2$, because $\diam(F_y) \leq \diam(F)$. \Cref{action of Gr} shows that there exists $g \in S_{d}$ such that $y^\prime = yg \in P$. Since $\{\vf_{y^\prime}\}_{y^\prime\in p^{-1}(x)}$ is a partition of unity on $M_2$, we deduce that $y \in Q_+$. This yields that
\[
|P \cap (\cap_{g \in S_{d}} Pg)| \leq |Q_+|,
\]
and thus,
\[
|Q_{+}| \geq |P| - | \cup_{g \in S_{d}} (P \smallsetminus Pg) | \geq
(1 - \ve | S_{d}|) |P|,
\]
where we used that $S_{d}$ is symmetric.

Consider now $y \in Q_-$ and $z \in F_y$ such that $0 < \chi_P(z) < 1$. Then there exist $y_1 \in P$ and $y_2 \in p^{-1}(x) \smallsetminus P$ such that $\vf_i (z) > 0$, $i=1,2$. This implies that $d(z,y_i) < r + 1$, and hence, $d(y,y_i) < d/2$, $i=1,2$. In particular, we obtain that $d(y_1,y_2) < d$ and there exists $g \in S_d$ such that $y_1 = y_2 g \in P \smallsetminus Pg$, in view of \Cref{action of Gr}. Since $d(y,y_1) < d/2$, we derive from \Cref{action of Gr} that for such a $y$ there exists at most $|S_{d/2}|$ such $y_1$. This gives the estimate
\[
|Q_-| \leq \sum_{g \in S_d} |S_{d/2}| |P \smallsetminus Pg| < \ve |S_{d/2}| |S_d| |P|.
\]
Combining the above, for $\ve > 0$ sufficiently small, we conclude that
\[
\frac{| Q_- |}{|Q_+|} < \frac{\ve |S_{d/2}| |S_d|}{1 - \ve |S_d|},
\]
which completes the proof. \qed
\end{proof}

\medskip

\noindent{\emph{Proof of \Cref{covering thm}:}} Suppose first that the covering $p \colon M_2 \to M_1$ is amenable. Since $M_1$ is compact, its Steklov spectrum is discrete. Therefore, for any $\lambda \in \sigma(M_1)$, there exists a harmonic $f \in C^{\infty}(M_1)$  with $\| f \|_{L^{2}(\partial M_1)} = 1$ such that $\nu_1(f) = \lambda f$ on $\partial M_1$, where $\nu_1$ is the outward pointing unit normal to $\partial M_1$. Denote by $\tilde{f} = f \circ p$ the lift of $f$ on $M_2$ and, given a a non-empty, finite subset $P$ of $p^{-1}(x)$, consider the function $f_P = \chi_P \tilde{f} \in C^{\infty}_c(M_2)$.

For $y \in Q_{+}$, keeping in mind (\ref{integration covering}) and that $f_P = \tilde{f}$ in $F_y$, we readily see that
\[
\int_{F_y} (\Delta f_P)^2 = 0 \text{, } \int_{\partial M_2 \cap F_y } ( \nu_2(f_P) - \lambda f_P)^2 = 0 \text{ and } \int_{\partial M_2 \cap F_y} f_P^2 = 1,
\]
where $\nu_2$ is the outward pointing unit normal to $\partial M_2$. Fix now $y \in Q_-$. Using (\ref{integration covering}) and that $\tilde{f}$ is harmonic, we compute
\begin{eqnarray}
\int_{F_y} (\Delta f_P)^2 &=& \int_{F_y} (\tilde{f} \Delta \chi_P - 2 \< \grad \tilde{f} , \grad \chi_P \>)^2 \nonumber \\
&\leq& 2 \int_{F_y} \tilde{f}^2 (\Delta \chi_P)^2 + 8 \int_{F_y} \| \grad \tilde{f} \|^2 \| \grad \chi_P \|^2 \nonumber\\
&\leq& 2C_2^2 \| f \|_{L^{2}(M_1)}^2 + 8 C_2^2 \int_{M_1} \| \grad f \|^2, \nonumber
\end{eqnarray}
where $C_2$ is the constant from (\ref{uniform estimates}). Moreover, we deduce that
\[
\int_{\partial M_{2} \cap F_y} (\nu_2(f_P) - \lambda f_p)^2 = \int_{\partial M_2 \cap F_y} (\nu_2(\chi_P)\tilde{f})^2 \leq C_2^2
\]
where we used that $\nu_2(\tilde{f}) = \lambda \tilde{f}$ on $\partial M_{2}$, $\| f \|_{L^{2}(\partial M_1)} = 1$ and (\ref{integration covering}). 

Furthermore, we derive from (\ref{integration covering}) and (\ref{intersections covering}) that
\[
\| f_P \|_{L^{2}(\partial M_{2})}^2 \geq \sum_{y \in Q_+} \int_{\partial M_{2} \cap F_y} f_P^2 = |Q_+|.
\]
Combining the above estimates, together with (\ref{intersections covering}) and \Cref{approximation of test functions}, yields that
\begin{eqnarray}
\| \nu(\mathcal{H}f_P) - \lambda f_P \|_{L^{2}(\partial M_2)}^2 &\leq& 2\| \nu(f_p) - \lambda f_p \|_{L^{2}(\partial M_2)}^2 + 2C^2 \| \Delta f_p \|_{L^{2}(M_2)}^2 \nonumber\\
&=& 2\sum_{y \in Q_{-}} \bigg( \int_{\partial M_{2} \cap F_y} (\nu(f_P) - \lambda f_P)^2 + C^2 \int_{F_y} (\Delta f_P)^2 \bigg) \nonumber \\
&\leq& 2C_2^2|Q_-| \big(1 + 2C^2 \| f \|_{L^{2}(M_1)}^2 + 8 C^2 \int_{M_1}\| \grad f \|^2\big).\nonumber
\end{eqnarray}
Since the constants involved in these estimates are independent from $P$, it follows from \Cref{amenability} that for any $\ve > 0$ there exists a finite subset $P$ of $p^{-1}(x)$ such that
\[
\| \nu(\mathcal{H}f_P) - \lambda f_P \|_{L^{2}(\partial M_2)}^2 < \ve \| f_P \|_{L^{2}(\partial M_{2})}^2.
\]
We conclude from \Cref{approximate eigenvalues} that $\lambda \in \sigma(M_2)$, $\ve > 0$ being arbitrary.

It is clear from \Cref{bottom of Steklov} that $\sigma_0(M_2) \geq 0$. Therefore, if $\sigma(M_1) \subset \sigma(M_2)$, then $\sigma_0(M_2) = 0$, because $\sigma_0(M_1) = 0$. Suppose now that $\sigma_0(M_2) = 0$. Then \Cref{Steklov and Neumann} states that $\lambda_0^N(M_2) = 0$, and thus, the covering is amenable, in view of \Cref{Brooks Neumann}. \qed

\section{Steklov spectrum under Riemannian submersions}\label{submersions}

The aim of this section is to establish \Cref{Submersion thm}. Throughout, we consider a Riemannian submersion $p \colon M_2 \to M_1$ arising from the action of a connected Lie group $G$, where $M_1$ is compact with boundary, and denote by $\nu_i$ the outward pointing unit normal to $\partial M_i$, $i=1,2$. For more details on Riemannian submersions, see \cite{MR2110043,MR1707341}. We begin by showing that $M_2$ has bounded geometry and $\lambda_0^D(M_2) > 0$.

\begin{proposition}\label{bounded geometry}
In the aforementioned setting, $M_2$ has bounded geometry.
\end{proposition}

\begin{proof}
Bearing in mind that $G$ acts on $M_2$ via isometries, it is easy to verify properties (ii)-(iv) of \Cref{manifold with bounded geometry}, after noticing that $B(\partial M_2, r) = p^{-1}(B(\partial M_1 , r))$ for any $r>0$. To check the validity of (i), observe that there exists $r > 0$ such that the map
	\[
	\exp \colon \partial M_1 \times [0,r) \to M_1, \text{ } (y,t) \mapsto \exp_y(-t \nu_1)
	\]
	is a diffeomorphism onto its image, $M_1$ being compact. Choose a precompact, open $U \subset \partial M_2$ such that $GU = \partial M_2$, and consider $r^\prime \leq r$ such that
	\[
	\exp \colon U \times [0,r^\prime) \to M_2, \text{ } (z,t) \mapsto \exp_z(-t \nu_2)
	\]
	is a diffeomorphism onto its image. Since $G$ acts on $M_2$ via isometries, we derive that
	\[
	\exp \colon \partial M_2 \times [0,r^\prime) \to M_2, \text{ } (z,t) \mapsto \exp_z(-t \nu_2)
	\]
	is a local diffeomorphism onto its image. Hence, it remains to show that this map is injective. 
	
	To this end, let $z_1,z_2 \in \partial M_2$ and $t_1,t_2 \in [0,r^\prime)$ with $\exp_{z_1}(-t_1 \nu_2) = \exp_{z_2}(-t_2 \nu_2) = z$. From the fact that $\gamma_i(t) = \exp_{z_i}(-t \nu_2)$ is a horizontal geodesic, we readily see that $(p \circ \gamma_i)(t) = \exp_{p(z_i)}(- t \nu_1)$, $i=1,2$. Since $r^\prime \leq r$, this yields that $p(z_1) = p(z_2)$ and $t_1 = t_2$. In particular, the geodesics $p \circ \gamma_i$ coincide, and hence, so do their horizontal lifts $\gamma_i$ with endpoint $z$. We conclude that $z_1 = z_2$ and $t_1 = t_2$, as we wished. \qed
\end{proof}\medskip

Consider the Schr\"odinger operator
\begin{equation}\label{Schrodinger}
S = \Delta + \frac{1}{4} \| p_* H \|^2 - \frac{1}{2} \diver p_*H
\end{equation}
on $M_1$. We know from \cite[Theorem 1.1]{Mine3} that $\lambda_0^D(M_2) \geq \lambda_0^D(S)$. It is worth to mention that \cite[Theorem 1.1]{Mine3} is formulated for manifolds without boundary, which however do not have to be complete. Hence, the assertion is readily extends to manifolds with boundary.

\begin{proposition}\label{S positive definite}
In the aforementioned setting, we have that $\lambda_0^D(S) > 0$.
\end{proposition}

\begin{proof}
	Given a non-zero $f \in C^{\infty}_{c}(M_1)$ with $f = 0$ on $\partial M_1$, we compute
	\begin{eqnarray}
		\< Sf , f \>_{L^{2}(M_1)} &=& \int_{M_1} (\| \grad f \|^2 + \frac{1}{4} \| p_*H \| f^2  - - \frac{1}{2} f^2 \diver p_*H) \nonumber \\
		&=& \int_{M_1} (\| \grad f \|^2 + \frac{1}{4} \| p_*H \| f^2  + \frac{1}{2} \< \grad f^2  , p_*H \>) \nonumber \\
		&=& \int_{M_1} \| \grad f  + \frac{f}{2}  p_*H \|^2,  \nonumber
	\end{eqnarray}
	where we used the divergence formula. In particular, we deduce that $\mathcal{R}_S(f) \geq 0$ for any such $f$, which means that $\lambda_0^D(S) \geq 0$, in view of \Cref{bottom Dirichlet}.
	
	Assume to the contrary that $\lambda_0^D(S) = 0$. Since $M_1$ is compact, the Dirichlet spectrum of $S$ is discrete. This yields that there exists $f \in C^{\infty}(M_1)$ positive in $M_1^\circ$ and vanishing on $\partial M_1$ such that $Sf = 0$, which implies that $\mathcal{R}_S(f) = 0$. From the preceding computation, we conclude that
	\[
	p_*H = - 2 \grad \ln f
	\]
	in $M_1^\circ$. This is a contradiction, since $p_*H$ is smooth on $M_1$, while $f$ vanishes on $\partial M_1$.\qed
\end{proof}\medskip

From the above, it follows that $M_2$ has bounded geometry and $\lambda_0^D(M_2)>0$. Therefore, we may define the Steklov spectrum of $M_2$ as in \Cref{Steklov spectrum}. 

The proof of \Cref{Submersion thm} relies on the methods of \cite{Mine3}. For convenience of the reader, we briefly discuss what will be used in the sequel. Given a section $s \colon U \subset M_1 \to M_2$, the map $\Phi \colon G \times U \to p^{-1}(U)$ defined by $\Phi(x,y) = xs(y)$ is a diffeomorphism. We denote by $g_{s(y)}$ the pullback of the Riemannian metric of the fiber $F_y = p^{-1}(y)$ via $\Phi(\cdot , y)$. Then $g_{s(y)}$ is a left-invariant metric depending smoothly on $y \in U$, according to \cite[Proposition 4.1]{Mine3}. The behavior of the volume elements of these metrics is illustrated in the following.

\begin{proposition}[{{\cite[Corollaries 4.2 and 4.3]{Mine3}}}]\label{volume functions}
Let $g$ be a fixed left-invariant metric on $G$. Given a section $s \colon U \subset M_1 \to M_2$, there exists $V_s \in C^{\infty}(U)$ such that the volume elements are related by
\[
d{\vol}_{g_{s(y)}} = V_s(y) d {\vol}_{g}.
\]
If, in addition, $G$ is unimodular, then there exists $V \in C^{\infty}(M)$ such that
\[
d{\vol}_{g_{s(y)}} = V(y) d {\vol}_{g}
\] 
for any section $s \colon U  \subset M_1 \to M_2$ and $y \in U$. Moreover, the gradient of $V$ is given by $\grad V = - V p_{*}H$.
\end{proposition}

In the case where $G$ is unimodular, it follows that the Schr\"odinger operator defined in (\ref{Schrodinger}) is written as
\[
S = \Delta - \frac{\Delta \sqrt{V}}{\sqrt{V}},
\]
and therefore, corresponds to the symmetric diffusion operator
\[
L = m_{\vf}^{-1} \circ S \circ m_{\vf} = \Delta - 2 \grad \ln \sqrt{V} = \Delta + p_{*}H.
\]
It is noteworthy that 
\begin{equation}\label{lift}
\Delta (f \circ p) = (Lf) \circ p
\end{equation}
for any $f \in C^{\infty}(M_1)$ (cf. for instance \cite[Lemma 2.6]{Mine3}). Moreover, since the Dirichlet spectra of $S$ and $L$ coincide ($m_{\vf}$ being an isometric isomorphism), we derive from \Cref{S positive definite} that $\lambda_0^D(L) > 0$. Hence, any $f \in C^{\infty}(M_1)$ admits a unique $L$-harmonic extension $\mathcal{H}_Lf \in C^{\infty}(M_1)$. This gives rise to the Dirichlet-to-Neumann map
\[
\Lambda_{L} \colon C^{\infty}(\partial M_1) \subset L^2_{\sqrt{V}}(\partial M_1) \to L^2_{\sqrt{V}}(\partial M_1), \text{ } f \mapsto \nu(\mathcal{H}_L f).
\]
It is standard that the spectrum $\sigma_L(M_1)$ of the Friedrichs extension of this map is discrete and the corresponding eigenfunctions are smooth.

Given a section $s \colon U \subset M_1 \to M_2$ and $f \colon G \to \R$, we denote by $f_s \colon p^{-1}(U) \to \R$ the function satisfying
\[
f_s(\Phi(x,y)) = f(x)
\]
for any $x \in G$ and $y \in U$.

\begin{proposition}[{{\cite[Lemma 4.6 and Proposition 4.7]{Mine3}}}]\label{cut-off functions}
Fix a left-invariant metric on $G$. Then for any $r>0$ and any bounded, open $W \subset G$, there exists $\chi \in C^{\infty}_c(G)$ with $\chi = 1$ in $W \smallsetminus B(\partial W,r)$, $\supp \chi \subset B(W,r/2)$, such that for any extensible section $s \colon U \to M_2$, there exists $C > 0$ independent from $W$, satisfying
\[
|\Delta (\chi_s) (z) | \leq C \text{ and } \| \grad (\chi_s)(z) \| \leq C
\]
for any $z \in p^{-1}(U)$.
\end{proposition}

The point of this proposition is that the constant depends only on the section, and not on the corresponding $W$. These functions are obtained from a partition of unity which is constructed by translates of a fixed function (conceptually related to the partition of unity of the previous section). The importance of this construction becomes more clear in the following consequence of \Cref{unimodular and amenable}, together with the Cheeger and Buser inequalities (more precisely, the main ingredient in the proof of the latter one).

\begin{proposition}[{{\cite[Corollary 2.11]{Mine3}}}]\label{domain volume}
Suppose that $G$ is non-compact, unimodular and amenable, and choose a left-invariant metric on it. Then for any $\ve > 0$ and $r > 0$, there exists a bounded, open $W \subset G$ such that
\[
| B(\partial W,r) | < \ve | W \smallsetminus B(\partial W , r) |.
\]
\end{proposition}

The most technical part of the proof of \Cref{Submersion thm} is contained in the following.

\begin{proposition}\label{pull-up submersions}
If $G$ is unimodular and amenable, then for any $\lambda \in \sigma_L(M_1)$ and $\ve > 0$, there exists $h \in C^{\infty}_c(M_2)$ such that $\| (\Lambda - \lambda) h \|_{L^2(\partial M_2)} < \ve \| h \|_{L^2(\partial M_2)}$.
\end{proposition}

\begin{proof}
Since $M_1$ is compact, for any $\lambda \in \sigma_{L}(M_1)$, there exists $f \in C^{\infty}(M_1)$ with $\| f \||_{L^2_{\sqrt{V}}(\partial M_1)} = 1$, $L f = 0$ in $M_1$ and $\nu_1(f) = \lambda f$ on $\partial M_1$. If $G$ is compact, it is immediate to verify that its lift $\tilde{f} = f \circ p$ is harmonic and satisfies $\nu_2 (\tilde{f}) = \lambda \tilde{f}$ on $\partial M_2$, by virtue of (\ref{lift}). Hence, it remains to prove the assertion in the case where $G$ is non-compact.
	
To this end, cover $M_1$ with finitely many open domains $U_i$ that admit extensible sections $s_i \colon U_i \subset M_1 \to M_2$, $i=1,\dots,k$, and consider a smooth partition of unity $\{\vf_i\}_{1\leq i \leq k}$ subordinate to $\{U_i\}_{1 \leq i \leq k}$. Denote by $x_{ij} \colon U_{i} \cap U_{j} \to G$ the transition maps, which are defined by $s_{j}(y) = x_{ij}(y)s_{i}(y)$ for all $y \in U_{i} \cap U_{j}$, and by $\Phi_{i} \colon G \times U_{i} \to p^{-1}(U_{i})$ the diffeomorphisms defined by $\Phi_{i}(x,y) = x s_{i}(y)$, $i,j=1,\dots,k$.

Choose a left-invariant metric $g$ on $G$. Using that $U_{i}$ is precompact and $s_{i}$ is extensible, we readily see that there exists $r > 0$ such that $x_{ij}(U_{i} \cap U_{j}) \subset B_{g}(e,r)$ for any $i,j=1,\dots,k$, where $e$ is the neutral element of $G$. Given a bounded, open $W \subset G$, denote by $\chi$ the corresponding function, according to \Cref{cut-off functions}, for $r$ as above, where we regard $G$ endowed with the fixed Riemannian metric $g$. Consider the compactly supported, smooth function
\[
h_{i} := \chi_{s_{i}} \tilde{\varphi}_{i} \tilde{f}
\]
in $p^{-1}(U_{i})$, $i=1 , \dots , k$, where $\tilde{\vf}_i = \vf_i \circ p$ and $\tilde{f} = f \circ p$. For $h = \sum_{i=1}^{k} h_{i}$, we obtain from \Cref{cut-off functions}, that there exists $C_f>0$ independent from $W$, such that $|h(z)| \leq C_f$, $\|\grad h(z)\| \leq C_f$ and $| \Delta h (z) | \leq C_f$ for any $z \in M_{2}$. It follows from \Cref{domain volume} that there exists a bounded, open $W \subset G$ such that
\begin{equation}\label{volume estimate}
	\frac{|W_{0}^{\prime}|_{g}}{|W_{0}|_{g}} < \min \left\{\frac{\ve^2}{8 (\lambda^2 + 1) C_f^{2} \int_{\partial M_1} V} , \frac{\ve^2}{4 C^2 C_f^2 \int_{M_1} V} \right\},
\end{equation}
where $W_0^\prime =  B(\partial W , 3r)$, $W_{0} = W \smallsetminus W_0^\prime$, and $C$ is the constant from \Cref{approximation of test functions} on $M_2$. To simplify the notation, set $D_0 = W \smallsetminus C(\partial W,2r)$, $D_0^\prime = C(W,2r)$, and given $y \in U_i$, let $W_i(y) = \Phi_i(W_0 , y)$, $W_i^\prime(y) = \Phi_i(W_0^\prime , y)$, $D_i(y) = \Phi_i(D_0 , y)$ and $D_i^\prime(y) = \Phi_i(D_0^\prime , y)$, $i=1,\dots,k$.
Here, $B(\cdot,\cdot)$ and $C(\cdot, \cdot)$ stand for open and closed tubular neighborhoods with respect to the fixed Riemannian metric $g$, respectively.
Using that
\[
\Phi_{i}(x,y) = \Phi_{j}(xx_{ji}(y) , y)
\]
for any $y \in U_{i} \cap U_{j}$ and $x \in G$, it is immediate to verify that $h(z) = \tilde{f}(z)$ for any $z \in D_{i}(y) \supset W_i(y)$ and that $\supp h \cap F_{y} \subset D_{i}(y) \cup D_{i}^{\prime}(y) \subset W_i(y) \cup W_i^\prime(y)$ for any $y \in U_{i}$, $i=1,\dots,k$. 

By virtue of \Cref{volume functions}, we compute
\begin{eqnarray}
	\| h \|_{L^{2}(\partial M_2)}^{2} &=& \sum_{i=1}^{k} \int_{\partial M_{2}} \tilde{\varphi}_{i}h^{2} \geq \sum_{i=1}^{k} \int_{\partial M_1 \cap U_{i}} \int_{W_{i}(y)} \tilde{\varphi}_{i}h^{2} \dy \nonumber\\
	&=& \sum_{i = 1}^{k} \int_{\partial M_1 \cap U_{i}} \varphi_{i}(y) f^{2}(y)|W_{0}|_{g_{s_{i}(y)}} \dy \nonumber \\
	&=& |W_{0}|_{g} \sum_{i = 1}^{k} \int_{\partial M_1 \cap U_{i}} \varphi_{i} f^{2} V = |W_{0}|_{g}, \nonumber
\end{eqnarray}
where we used that $\| f \|_{L^2_{\sqrt{V}}(\partial M_1)} = 1$.
Moreover, it is evident that
\[
\| \nu_2 (h) - \lambda h \|_{L^{2}(\partial M_{2})}^{2} = \sum_{i = 1}^{k} \int_{\partial M_{2}} \tilde{\varphi}_{i} (\nu_2 (h) - \lambda h)^{2} = Q_1 + Q_2,
\]
where
\[
Q_1 = \sum_{i= 1}^{k}  \int_{\partial M_1 \cap U_{i}} \int_{W_{i}(y)} \tilde{\varphi}_{i} (\nu_2 (h) - \lambda h)^{2} \dy , \text{ } Q_2 = \sum_{i= 1}^{k} \int_{\partial M_1 \cap U_{i}} \int_{W_{i}^{\prime}(y)} \tilde{\varphi}_{i} (\nu_2 (h) - \lambda h)^{2} \dy.
\]
To estimate these quantities, recall that $D_0$ is an open subset of $G$, and hence, so is $\Phi_i(D_0 \times U_i) \subset M_2$, which is a neighborhood of $W_i(y)$ for any $y \in U_i \cap \partial M_1$, $i=1,\dots,k$. Since $h = \tilde{f}$ in $\Phi_i(D_0 \times U_i)$, in $W_i(y)$ we have that
\[
\nu_2(h) = \nu_2(\tilde{f}) = \< \nu_2 , \grad \tilde{f} \> = \< \nu_1 , \grad f \> \circ p = \lambda \tilde{f} = \lambda h,
\]
and thus, $Q_1 = 0$. In $W_i^\prime(y)$, using \Cref{volume functions} and that $(\nu_2(h) - \lambda h)^2 \leq 2C_f^2(\lambda^2 + 1)$, we readily see that
\[
Q_2 \leq 2C_f^2(\lambda^2 + 1) \sum_{i=1}^{k} \int_{\partial M_1 \cap U_{i}} \varphi_{i}(y) |W^{\prime}_{0}|_{g_{s_{i}(y)}} \dy 
= 2C_f^2(\lambda^2 + 1) |W^{\prime}_{0}|_{g} \int_{\partial M_1} V.
\]

Furthermore, it is apparent that
\[
\| \Delta h \|_{L^{2}(M_{2})}^{2} = \sum_{i = 1}^{k} \int_{M_{2}} \tilde{\varphi}_{i} (\Delta h)^{2} = Q_3 + Q_4,
\]
where
\[
Q_3 = \sum_{i= 1}^{k} \int_{U_{i}} \int_{W_{i}(y)} \tilde{\varphi}_{i} (\Delta h)^{2} \dy, \text{ } Q_4 = \sum_{i=1}^k \int_{U_{i}} \int_{W_{i}^{\prime}(y)} \tilde{\varphi}_{i} (\Delta h)^{2} \dy. \nonumber
\]
Using again that $h = \tilde{f}$ in $\Phi_i(D_0 \times U_i) \supset W_i(y)$, we obtain from (\ref{lift}) that $\Delta h = 0$ in $W_i(y)$, and hence, $Q_3 = 0$. Finally, \Cref{volume functions} implies that
\[
Q_4 \leq C_f^{2} \sum_{i=1}^{k} \int_{U_{i} } \varphi_{i}(y) |W^{\prime}_{0}|_{g_{s_{i}(y)}} dy 
= C_f^{2} |W^{\prime}_{0}|_{g} \int_{M_1} V. 
\]

From the above estimates and \Cref{approximation of test functions}, we conclude that
\[
\frac{\| (\Lambda - \lambda)h \|^2_{L^{2}(\partial M_{2})}}{\| h \|^2_{L^{2}(\partial M_{2})}} \leq 2 \frac{\| \nu(h) - \lambda h \|^2_{L^{2}(\partial M_{2})}}{\| h \|^2_{L^{2}(\partial M_{2})}} + 2C^2 \frac{\| \Delta h \|^2_{L^2(M_2)}}{\| h \|^2_{L^{2}(\partial M_2)}} < \ve^2,
\]
by virtue of (\ref{volume estimate}). \qed\medskip
\end{proof}

Another important ingredient in the proof of \Cref{Submersion thm} involves the behavior of the Neumann spectrum under Riemannian submersions. One can establish the following by arguing as in \cite{Mine3}. However, in our setting, where the base manifold is compact, we can prove it in a simpler way (which also establishes the analogous assertion if the base manifold is closed).

\begin{theorem}\label{converse Neumann}
If $\lambda_0^N(M_2) = 0$, then $G$ is unimodular and amenable.
\end{theorem}

\begin{proof}
	Cover $M_1$ with finitely many open domains $U_i$ admitting extensible sections $s_i  \colon U_i \subset M_1 \to M_2$, and denote by $\Phi_{i,y} \colon G \to F_y$ the diffeomorphism $\Phi_{i,y}(x) = x s_i(y)$ with $y \in U_i$, $i=1,\dots,k$. Let $\{\vf_{i}\}_{1\leq i\leq k}$ be a smooth partition of unity subordinate to $\{U_i\}_{1\leq i\leq k}$.
	Fix a left-invariant Riemannian metric $g$ on $G$ and consider the functions $V_{s_i} \in C^{\infty}(U_i)$ from \Cref{volume functions}. Since $s_i$ is extensible, it follows that there exists $c>0$ such that 
	\[
	\| {\grad}_{g_{s_i(y)}} f \|_{g_{s_i(y)}} \geq c \| {\grad}_{g} f \|_{g}
	\]
	for any $f \in C^{\infty}(G)$ and $y \in U_i$, $i=1,\dots,k$.

	Since $\lambda_0^N(M_2) = 0$, we obtain from \Cref{bottom Neumann} that for any $\ve > 0$ there exists a non-zero $f \in C^{\infty}_c(M_2)$ such that
	\begin{eqnarray}
		\ve &>& \frac{\int_{M_2} \| \grad f \|^2}{\int_{M_2} f^2} \geq \frac{\int_{M_2}  \| (\grad f)^v \|^2}{\int_{M_2} f^2} = \frac{\int_{M_1} \int_{F_y} \| \grad (f|_{F_y}) \|^2 \dy}{\int_{M_1} \int_{F_y} f^2 \dy}\nonumber\\
		&=& \frac{ \sum_{i=1}^k \int_{U_i} \vf_{i}(y) \int_{F_y} \| \grad (f|_{F_y}) \|^2 \dy}{\sum_{i=1}^k \int_{U_i} \vf_{i}(y) \int_{F_y} f^2 \dy} \nonumber\\
		&=&  \frac{\sum_{i=1}^k \int_{U_i} \vf_{i}(y) \int_{G} \| \grad_{g_{s_i(y)}} (f \circ \Phi_{i,y}) \|_{g_{s_i(y)}}^2 V_{s_i}(y) \dy}{\sum_{i=1}^k \int_{U_i} \vf_{i}(y) \int_{G} (f \circ \Phi_{i,y})^2 V_{s_i}(y) \dy} \nonumber
	\end{eqnarray}
where $(\grad f)^v$ stands for the vertical component of $\grad f$, and the integrals over $G$ are with respect to the fixed Riemannian metric $g$. It is now clear that there exists $1 \leq i \leq k$ and $y \in U_i$ such that $\vf_{i}(y) > 0 $, $f$ is not identically zero on $F_y$, and we have that
\[
\ve > \frac{ \vf_i(y) \int_{G}  \| \grad_{g_{s_i(y)}} (f \circ \Phi_{i,y}) \|_{g_{s_i(y)}}^2 V_{s_i}(y) }{\vf_i(y) \int_{G}  (f \circ \Phi_{i,y})^2 V_{s_i}(y) } \geq c^2  \frac{\int_{G}  \| \grad_{g} (f \circ \Phi_{i,y}) \|_{g}^2  }{\int_{G}  (f \circ \Phi_{i,y})^2  } = \mathcal{R}_{g}(f \circ \Phi_{i,y}).
\]
Since $\ve > 0$ is arbitrary, we conclude from \Cref{bottom no boundary} that $\lambda_0(G,g) = 0$, and therefore, $G$ is unimodular and amenable, by \Cref{unimodular and amenable}.\qed
\end{proof}

\medskip

\noindent{\emph{Proof of \Cref{Submersion thm}:}} 
If $G$ is unimodular and amenable, then $\sigma_L(M_1) \subset \sigma(M_2)$ by virtue of \Cref{pull-up submersions,,approximate eigenvalues}. Taking into account that $0 \in \sigma_L(M_1)$, it evident that the second statement implies the third. Finally, if $\sigma_0(M_2) = 0$, then we derive from \Cref{Steklov and Neumann} that $\lambda_0^N(M_2) =0$, and hence $G$ is unimodular and amenable, in view of \Cref{converse Neumann}. \qed

\noindent Max Planck Institute for Mathematics \\
Vivatsgasse 7, 53111, Bonn \\
E-mail address: polymerp@mpim-bonn.mpg.de

\end{document}